\documentclass[12pt,twoside,reqno]{amsart}
\usepackage[utf8]{inputenc}
\usepackage[colorlinks=true, citecolor=blue, linkcolor=blue]{hyperref}
\usepackage{mathptmx, amsmath, amssymb, amsfonts, amsthm, mathptmx, enumerate, color,mathrsfs}
\usepackage{graphicx}
\usepackage{tikz}
\usepackage{lineno}
\usepackage{epstopdf}
\newcommand{\lap}{\mbox{$\Delta$}}
\newtheorem{theorem}{Theorem}[section]
\newtheorem{lemma}[theorem]{Lemma}

\newtheorem{corollary}[theorem]{Corollary}

\newcommand{\la}{\lambda}
\newcommand{\be}{\begin{equation}}
\newcommand{\ee}{\end{equation}}

\theoremstyle{definition}

\newtheorem{remark}{Remark}
\numberwithin{equation}{section}
\allowdisplaybreaks[4]

\begin{document}
\setcounter{page}{1}

\vspace*{2.0cm}
\title[ Pointwise arbitrarily  high-order interior estimates ]
{Pointwise arbitrarily  high-order interior estimates for mixed local and nonlocal elliptic equations}
\author[ Pengyan Wang,   Leyun Wu, Chilin Zhang]{ Pengyan Wang$^a$, Leyun Wu$^b$, Chilin Zhang$^{c}$\textsuperscript{$ \ast $}}\thanks{$^{\ast}$Corresponding author.}
\maketitle
\vspace*{-0.6cm}

\begin{center}
{\footnotesize

$^a$   School of Mathematics and Statistics, Xinyang Normal University, Xinyang,   China\\

$^b$   School of Mathematics,  South China University of Technology,  Guangzhou,  China \\

$^c$   School of Mathematical Sciences, Fudan University, Shanghai,  China

}\end{center}

\vskip 4mm {\footnotesize \noindent {\bf Abstract.} In this paper, we mainly focus on 
pointwise, arbitrarily high-order interior estimates for the mixed local–nonlocal elliptic equation
\begin{equation*}
    (-\Delta)^su(x)-\Delta u(x)=f(x),\quad x\in B_r(0)
\end{equation*}
with $0<s<1$. 
The main challenges in this setting are the absence of an explicit Green function and the ineffectiveness of standard bootstrap arguments.

Our approach overcomes these difficulties via the Campanato iteration method, which inductively constructs polynomials approximating the solution. Using the fact that all functions are locally $s$-harmonic up to a small error significantly reduces the computational complexity when studying the regularity of the fractional Laplacian acting on these polynomials. Finally, the regularity estimates are obtained from the analysis of a manageable recursive inequality system.


}

\renewcommand{\thefootnote}{}

\footnotetext{E-mail addresses: wangpy@xynu.edu.cn (P. Wang); leyunwu@scut.edu.cn (L. Wu); zhangchilin@fudan.edu.cn (C. Zhang).}

\section{Introduction}

The study of equations driven by mixed local and nonlocal operators has attracted considerable attention in recent years, largely motivated by models in physics, finance, and biology where both the classical random walk and the L'{e}vy flight coexist; see \cite{DV} and the references therein. For instance, Su et al. \cite{SVWZ1} obtained the boundedness of any weak solution for mixed local and nonlocal elliptic equations via the Moser iteration method. Wang and Wu \cite{WW} studied Liouville theorems for mixed local and nonlocal elliptic and parabolic equations with indefinite nonlinearities. Dipierro et al. \cite{DSVZ} established the existence of positive solutions for a mixed-order nonlinear Schr"odinger equation. Very recently, Chen and Wu \cite{CW} obtained refined regularity results for mixed local–nonlocal equations.

For the pure fractional Laplacian $(-\lap)^s$, a comprehensive regularity theory has been developed over the past two decades. The foundational extension technique of Caffarelli and Silvestre \cite{CS} recasts $(-\lap)^s$ in $\mathbb R^n$ as a Dirichlet-to-Neumann map for a degenerate (but local) elliptic equation in $\mathbb{R}^{n+1}_{+}$, thereby opening the door to purely local PDE methods. Based on the extension method, Caffarelli and Silvestre established H\"{o}lder regularity, Harnack inequalities, and a complete Schauder theory for fractional equations. Schauder estimates for nonlocal fully nonlinear equations were developed by Jin  and Xiong \cite{JX}, who proved that if $f\in C^\alpha$ and if $u$ solves a nonlocal fully nonlinear elliptic equation with a translation-invariant kernel, then $u\in C^{2s+\alpha}$ in the interior. Their work extended the classical $C^{2,\alpha}$ theory of Caffarelli \cite{C} for local fully nonlinear equations (see also the monograph of Caffarelli and Cabr\'{e} \cite{CC}) to the nonlocal realm. Ros-Oton and Serra \cite{RS} subsequently developed an exhaustive boundary regularity theory for the Dirichlet problem associated with the fractional Laplacian, proving optimal $C^s$ regularity up to the boundary and providing the first complete Schauder estimates in H\"{o}lder spaces. Recently, Li and Wu \cite{LW} developed a pointwise Schauder estimate for the pure fractional equation $(-\lap)^su=f$. By constructing a sequence of Taylor-type approximation polynomials, they established arbitrarily high-order interior regularity of the solution.

For the mixed operator $(-\lap)^s-\lap$, the existing literature is substantially sparser. Su et al. \cite{SVWZ1} obtained the interior $L^\infty$, $C^{1,\alpha}$, and $W^{2,p}$ estimates for solutions with zero exterior conditions. Later in \cite{SVWZ}, the same group further established interior and boundary $C^{1,\alpha}$ and $C^{2,\alpha}$ estimates, where the exterior condition is still zero.

Compared to the fractional Laplacian, the mixed local and nonlocal operator lacks a precise Green function. Although sharp Green function estimates for $(-\lap)^s-\lap$ in a $C^{1,1}$ open set were established by Chen et al. \cite{CKSV}, these results are not sufficient to establish higher-order Schauder interior estimates for mixed local and nonlocal equations in bounded domains. The superposition of the classical Laplacian and the fractional Laplacian destroys the homogeneity that characterizes each component separately. Because the two operators scale differently under dilations, the resulting equation is not scale‑invariant, and no simple integral representation is available for its solutions. This fundamental obstruction renders many classical approaches inapplicable: blow‑up arguments become delicate, perturbation techniques fail to decouple the two terms, and the bootstrap machinery that works so efficiently for purely local or purely nonlocal equations cannot be directly transferred. Our approach addresses these challenges by employing the Campanato iteration method, which inductively constructs polynomials that approximate the solution.

An alternative approach is to use the Campanato iteration, which was originally established by Campanato \cite{Co} to study Schauder estimates for solutions to second-order elliptic equations, and was later developed by Avellaneda and Lin \cite{AL} to address perturbed second-order elliptic equations. The Campanato iteration relies less on the Green function, so it is more robust and can be applied to fully non-linear equations. For example, in \cite{C1}, Caffarelli used a similar iteration method to obtain the Schauder estimate (as well as $W^{2,p}$ estimate) of the Monge-Amp\`{e}re equation, where the approximation functions are solutions to $\det D^{2}u\equiv1$. For the linearized Monge-Amp\`{e}re equation with rough data, see Goldman, Huesmann and Otto \cite{GHO}.

Throughout this paper, we denote by $B_r(x)$ the ball centered at $x$ with radius $r$. When $x=0$, we abbreviate $B_r(0)$ as $B_r$.

In this paper, we mainly investigate pointwise arbitrarily  high-order interior estimates of the following mixed local and nonlocal elliptic equation (with $r\leq 1$ and $0<s<1$)
\begin{equation}\label{7}
(-\lap)^su(x)-\lap u(x)= f(x),\quad x \in B_r.
\end{equation}
It is commonly acknowledged that the usual fractional Laplacian is defined by a singular integral
\begin{equation}\label{fL}
(-\Delta)^s u(x)=C_{n, s}P.V. \int_{\mathbb{R}^n}\frac{u(x)-u(y)}{|x-y|^{n+2s}}dy=C_{n, s} \lim_{\epsilon \to 0}\int_{\mathbb{R}^n\backslash B_\epsilon(x)}\frac{u(x)-u(y)}{|x-y|^{n+2s}}dy.
\end{equation}
Here, $P.V.$ stands for the Cauchy principal value, and $C_{n, s}$ is a dimensional constant depending on $n$ and $s$, precisely given by
\begin{equation*}
    C_{n, s}=\frac{ 4^s s\Gamma(\frac{n+2s}{2})}{\pi^{\frac{n}{2}}\Gamma(1-s)},
\end{equation*}
where $\Gamma$ denotes the Gamma function. 
It is easy to verify that the integral on the right hand side of \eqref{fL} is well defined
if  $u\in C_{loc}^{1, 1}(\mathbb{R}^n)\cap \mathcal{L}_{2s}$, where
$$
\mathcal{L}_{2s}=\left\{u\in L_{loc}^1 \,\Big| \int_{\mathbb{R}^n}\frac{|u(x)|}{1+|x|^{n+2s}}dx<\infty \right\}
$$
endowed naturally with the norm
$$
\|u\|_{\mathcal{L}_{2s}}:= \int_{\mathbb{R}^n}\frac{|u(x)|}{1+|x|^{n+2s}}dx.
$$

\begin{theorem}\label{t1}
Let $u\in C^2(B_r)\cap C^0(\overline{B_r})\cap \mathcal{L}_{2s}$ be a solution of \eqref{7} for some $r\leq1$. For any $m\geq0$, assume that there exists an $m$-order polynomial $F(x)$, such that
\begin{equation*}
    \|f-F\|_{L^{\infty}(B_{\rho})}\leq\omega(\rho)\cdot\rho^{m},
\end{equation*}
where $\omega(\rho)$ is a bounded function for $\rho\in[0,r]$.
\begin{itemize}
    \item[(1)] If $m\geq0$, then there exists an $(m+1)$-order polynomial $P(x)$, such that
    \begin{align*}
        &\|P\|_{L^{\infty}(B_{r})}+\sup_{\rho\leq r}\Big(\frac{\|u-P\|_{L^{\infty}(B_{\rho})}}{(\frac{\rho}{r})^{m+2}\ln{(\frac{2r}{\rho})}}\Big)\\
        \leq&C(n,s,m)\Big\{\|u\|_{L^{\infty}(B_{r})}+r^{2}\int_{B_{r}^{c}}\frac{|u(x)|}{|x|^{n+2s}}dx+r^{2}\|F\|_{L^{\infty}(B_{r})}+r^{m+2}\sup_{\rho\leq r}\omega(\rho)\Big\}.
    \end{align*}
    \item[(2)] If $m\geq0$ and $\displaystyle\int_{0}^{r}\frac{\omega(t)}{t}dt<\infty$, then there exists a uniform constant $\mu>0$ and an $(m+2)$-order polynomial $P(x)$, such that
    \begin{equation*}
        \|P\|_{L^{\infty}(B_{r})}\leq C(n,s,m)\Big\{\|u\|_{L^{\infty}(B_{r})}+r^{2}\int_{B_{r}^{c}}\frac{|u(x)|}{|x|^{n+2s}}dx+r^{2}\|F\|_{L^{\infty}(B_{r})}+r^{m+2}\int_{0}^{r}\frac{\omega(t)}{t}dt\Big\},
    \end{equation*}
    and
    \begin{align*}
        \|u-P\|_{L^{\infty}(B_{\rho})}\leq&C(n,s,m)(\frac{\rho}{r})^{m+2+\mu}\cdot\Big\{\|u\|_{L^{\infty}(B_{r})}+r^{2}\int_{B_{r}^{c}}\frac{|u(x)|}{|x|^{n+2s}}dx+r^{2}\|F\|_{L^{\infty}(B_{r})}\Big\}\\
        &+C(n,s,m)(\frac{\rho}{r})^{m+2}\cdot\Big\{r^{m+2}\int_{0}^{\rho}\frac{\omega(t)}{t}dt+r^{m+2}\rho^{\mu}\int_{\rho}^{r}\frac{\omega(t)}{t^{1+\mu}}dt\Big\}.
    \end{align*}
    \item[(3)] If $m\geq0$ and $\omega(\rho)=O(\rho^{\alpha})$ for some $\alpha\in(0,1)$, then there exists an $(m+2)$-order polynomial $P(x)$, such that
    \begin{align*}
        &\|P\|_{L^{\infty}(B_{r})}+\sup_{\rho\in[0,r]}\Big((\frac{r}{\rho})^{m+2+\alpha}\|u-P\|_{L^{\infty}(B_{\rho})}\Big)\\
        \leq&C(n,s,\alpha,m)\Big\{\|u\|_{L^{\infty}(B_{r})}+r^{2}\int_{B_{r}^{c}}\frac{|u(x)|}{|x|^{n+2s}}dx+r^{2}\|F\|_{L^{\infty}(B_{r})}+r^{m+2+\alpha}\cdot\sup_{t\in[0,r]}\frac{\omega(t)}{t^{\alpha}}\Big\}.
    \end{align*}
\end{itemize}
\end{theorem}
\begin{remark}
\begin{itemize}
    \item[(i)] Our results hold without any restrictions on the exterior condition imposed in $B_{r}^{c}$ (e.g., zero exterior condition).
    \item[(ii)] The coefficient of the integral $\int_{B_{r}^{c}}\frac{|u(x)|}{|x|^{n+2s}}dx$ here exhibits a smaller $r^2$ instead of $r^{2s}$, because the integer-order Laplacian dominates when the radius $r$ is very small. 
\end{itemize}
\end{remark}

As a consequence of Theorem~\ref{t1}, we obtain higher-order interior Schauder estimates for mixed local–nonlocal equations of the form \eqref{7}. Before stating the result, we first introduce some notation and function spaces.

For an integer $k \geq 0$, the space $C^k(\overline {B_r})$ consists of functions whose derivatives up to order $k$ are continuous in $\overline {B_r}$, endowed with the norm
\[
\|f\|_{C^k(\overline {B_r})} := \sum_{|\beta| \leq k} \sup_{x \in {B_r}} |D^\beta f(x)|.
\]

For $\alpha \in (0,1)$ and an integer $k\geq 0$, the space
$C^{k,\alpha}(\overline{B_r})$ consists of functions
$f\in C^k(\overline{B_r})$ such that all derivatives of order $k$
are $\alpha$-H\"older continuous, namely,
\[
|D^\beta f(x)-D^\beta f(y)|
\leq C|x-y|^\alpha,
\qquad
\forall x,y\in B_r,\ x\neq y,
\]
for all multi-indices $\beta$ satisfying $|\beta|=k$.
The corresponding norm is defined by
\begin{equation*}
\|f\|_{C^{k,\alpha}(\overline{B_r})}
:=
\|f\|_{C^k(\overline{B_r})}
+
\sum_{|\beta|=k}
\sup_{\substack{x,y\in B_r\\x\neq y}}
\frac{|D^\beta f(x)-D^\beta f(y)|}
{|x-y|^\alpha}.
\end{equation*}

The Ln-Lipschitz space $C^{k,\ln L}(\overline {B_r})$ is defined as the set of functions $f \in C^k(\overline {B_r})$ such that all derivatives of order $k$ satisfy a logarithmic Lipschitz continuity condition, namely
\[
|D^\beta f(x) - D^\beta f(y)| \leq C\, |x-y| \, \big|\ln \min(|x-y|,1/2)\big|, \quad \forall x,y \in {B_r},\ x \neq y,
\]
for all multi-indices $\beta$ with $|\beta|=k$. The corresponding norm is given by
\begin{equation*}
\|f\|_{C^{k,\ln L}(\overline {B_r})}
= \|f\|_{C^{k}(\overline {B_r})}
+ \sum_{|\beta| = k} \sup_{\substack{x,y \in {B_r} \\ x \neq y}}
\frac{|D^\beta f(x) - D^\beta f(y)|}{|x-y| \, \big|\ln \min(|x-y|,1/2)\big|}.
\end{equation*}

\medskip

Similarly, 
the Dini space $C^{k,\mathrm{Dini}}(\overline{B_r})$
consists of functions $f\in C^k(\overline{B_r})$ such that all derivatives of order $k$
are Dini continuous. More precisely, for each multi-index $\beta$
with $|\beta|=k$, define the modulus of continuity
\[
\omega_{D^\beta f}(t)
:=
\sup\big\{
|D^\beta f(x)-D^\beta f(y)|:
x,y\in B_r,\ |x-y|\leq t
\big\}.
\]
We say that $f\in C^{k,\mathrm{Dini}}(\overline{B_r})$
if
\[
\int_0^{\mathrm{diam}(B_r)}
\frac{\omega_{D^\beta f}(t)}{t}\,dt
<\infty
\]
for every multi-index $\beta$ satisfying $|\beta|=k$.
The corresponding norm is defined by
\begin{equation*}
\|f\|_{C^{k,\mathrm{Dini}}(\overline{B_r})}
:=
\|f\|_{C^k(\overline{B_r})}
+
\sum_{|\beta|=k}
\int_0^{\mathrm{diam}(B_r)}
\frac{\omega_{D^\beta f}(t)}{t}\,dt.
\end{equation*}

\begin{corollary}
Let $0<s<1$, $0<\alpha<1$, $r\leq1$, and $m\geq0$. Assume that $u\in C^2(B_r)\cap C^0(\overline{B_r})\cap \mathcal{L}_{2s}$ is a solution to the equation \eqref{7}. Then the following conclusions hold:
\begin{itemize}
    \item[(1)] If $f\in C^m(B_r)$, then $u\in C^{m+1,ln L}(B_{r/2})$.
    \item[(2)] If $f\in C^{m,\text{Dini}} (B_r) $, then $ u \in C^{m+2} (B_{r/2})$.
    \item[(3)] If $f\in C^{m,\alpha}(B_r)$, then $ u\in C^{m+2,\alpha}(B_{r/2})$.
\end{itemize}\end{corollary} 

In this paper, the basic idea of using the Campanato iteration is as follows: First, we approximate solutions to \eqref{7} using solutions to the Dirichlet problem \eqref{eq. Dirichlet problem}, which involves a purely local equation. Then, we use $(m+2)$-order polynomials to approximate solutions to \eqref{eq. Dirichlet problem}.

We would like to emphasize that we can obtain arbitrarily high-order regularity for $(-\Delta)^{s}\widetilde{Q}$, where $\widetilde{Q}$ is a polynomial near the origin (see Lemmas~\ref{lem. very general (-Delta)^s osc estimate} and~\ref{lem. pointwise s effect on smooth function}). Compared to \cite{SVWZ}, our computation is simpler, and the resulting estimate is significantly stronger. The key idea relies on the important observation in \cite{DSV} that every function is locally $s$-harmonic up to a small error of arbitrarily high order.

This paper is organized as follows. Section 2 collects preliminary facts, including estimates of the fractional Laplacian of cut-off polynomials  and the construction of the barrier function. Section 3 develops the core Campanato iteration: we define the sequence of approximating polynomials $P_k$, the auxiliary Dirichlet problems, and the key quantities $\varepsilon_k$ and $ \mathcal{O}_k$. Lemmas 3.1-3.5 provide the essential estimates needed for the recursion, culminating in Lemma 3.6. Finally, we provide a proof of Theorem \ref{t1}.

\section{Preliminaries}
In this section, we prove some preliminary facts.
\subsection{Fractional Laplacian of quadratic polynomials}
First, we estimate the effect of applying $(-\Delta)^{s}$ to a smooth function.
\begin{lemma}\label{lem. very general (-Delta)^s osc estimate}
    Let $Q_{1}$ and $Q_{2}$ be two $(m+2)$-order polynomials for $m\geq0$. Let $\Omega$ be a bounded region satisfying $B_{R}\subseteq\Omega$ for some $R>0$. Let
    \begin{equation*}
        \widetilde{Q}_{1}(x)=Q_{1}(x)\cdot\chi_{B_{R}}+v(x)\cdot\chi_{B_{R}^{c}},\quad\widetilde{Q}_{2}(x)=Q_{2}(x)\cdot\chi_{\Omega}+v(x)\cdot\chi_{\Omega^{c}},
    \end{equation*}
    where $v\in\mathcal{L}_{2s}$ is an arbitrary function. Let us denote
    \begin{equation*}
        Q(x)=Q_{1}(x)-Q_{2}(x),\quad\widetilde{Q}(x)=\widetilde{Q}_{1}(x)-\widetilde{Q}_{2}(x).
    \end{equation*}
    Then, there exists a uniform constant $C=C(n,s,m)$, such that
    \begin{equation*}
        R^{l}\Big\|D^{l}(-\Delta)^{s}\widetilde{Q}\Big\|_{L^{\infty}(B_{R/2})}\leq\frac{C}{R^{2s}}\cdot\Big(\|Q\|_{L^{\infty}(B_{R})}+\|v-Q_{2}\|_{L^{\infty}(\Omega)}\Big),\quad\mbox{for all }0\leq l\leq m+1.
    \end{equation*}
\end{lemma}
\begin{proof}
    It suffices to prove Lemma~\ref{lem. pointwise s effect on smooth function} below, which is a pointwise version of Lemma~\ref{lem. very general (-Delta)^s osc estimate}.
\end{proof}
\begin{lemma}\label{lem. pointwise s effect on smooth function}
    Under the assumptions of Lemma~\ref{lem. very general (-Delta)^s osc estimate}, for all $\delta\leq\frac{R}{4}$, there exists a $(m+1)$-order polynomial $P_{\delta}(x)$ satisfying
    \begin{equation*}
        \|P_{\delta}\|_{L^{\infty}(B_{R})}+(\frac{R}{\delta})^{m+2}\|(-\Delta)^{s}\widetilde{Q}-P_{\delta}\|_{L^{\infty}(B_{\delta})}\leq\frac{C(n,s,m)}{R^{2s}}\cdot\Big(\|Q\|_{L^{\infty}(B_{R})}+\|v-Q_{2}\|_{L^{\infty}(\Omega)}\Big).
    \end{equation*}
\end{lemma}
Its proof relies on an important fact in \cite{DSV}: all functions are almost $s$-harmonic, as stated below.
\begin{lemma}\label{lem. all functional are s-harmonic}
    Let $R>0$ be arbitrary. Let $M\in\mathbb{N}$, and let $I(M)$ be the index set defined as follows:
    \begin{equation*}
        I(M):=\Big\{\vec{\beta}=(\beta_{1},\cdots,\beta_{n}):\ \beta_{i}\in\mathbb{Z}_{\geq0}\mbox{ and }\sum_{i=1}^{n}\beta_{i}\leq M\Big\}.
    \end{equation*}
    Then, there exist sufficiently large constants $\Lambda_{0}=\Lambda_{0}(s,M)$ and $C_{0}(s,M)$ such that, for every $\vec{\beta}\in I(M)$, there exists a function
    \begin{equation*}
        H_{\vec{\beta}}(x)\in C^{M+1}(B_{R})\cap C_{c}(B_{\Lambda_{0}R})\quad\mbox{satisfying }\|H_{\vec{\beta}}\|_{L^{\infty}(\mathbb{R}^{n})}\leq C_{0}.
    \end{equation*}
Moreover, $ H_{\vec{\beta}}$ is s-harmonic in $B_{R}$, namely,
\[
(-\Delta)^{s}H_{\vec{\beta}}(x)=0
\qquad \text{in } B_{R},
\]
and admits the following asymptotic expansion in $B_{R}$: 
    \begin{equation*}
        H_{\vec{\beta}}(x)=\prod_{i=1}^{n}(\frac{x_{i}}{R})^{\beta_{i}}+O(|x|^{M+1}),\quad\mbox{and }\,\|D^{M+1}H_{\vec{\beta}}\|_{L^{\infty}(B_{R})}\leq\frac{C_{0}}{R^{M+1}}.
    \end{equation*}
\end{lemma}

\begin{proof}
   This lemma follows from \cite[Theorem 3.1]{DSV} through the scaling technique.
\end{proof}
Now, we are able to prove Lemma~\ref{lem. pointwise s effect on smooth function}.
\begin{proof}[Proof of Lemma~\ref{lem. pointwise s effect on smooth function}]
    We have the following decomposition of $Q(x)$:
    \begin{equation*}
        Q(x)=\sum_{\vec{\beta}\in I(m+2)}a_{\vec{\beta}}\cdot\prod_{i=1}^{n}(\frac{x_{i}}{R})^{\beta_{i}},\quad\mbox{where }|a_{\vec{\beta}}|\leq C\|Q\|_{L^{\infty}(B_{R})}.
    \end{equation*}
    Using the functions $H_{\vec{\beta}}(x)$ in Lemma~\ref{lem. all functional are s-harmonic} with $M=m+3$, we set
    \begin{equation*}
        H(x)=\sum_{\vec{\beta}\in I(m+2)}a_{\vec{\beta}}\cdot H_{\vec{\beta}}(x),\quad\mathcal{E}(x)=Q(x)\cdot\chi_{B_{R}}-H(x).
    \end{equation*}
    Then, we must have $(-\Delta)^{s}\Big(\widetilde{Q}(x)\cdot\chi_{B_{R}}\Big)\equiv(-\Delta)^{s}\mathcal{E}(x)$ in $B_{R}$. Besides,
    \begin{equation*}
        \|\mathcal{E}\|_{L^{\infty}(\mathbb{R}^{n})}\leq C\|Q\|_{L^{\infty}(B_{R})},\quad\mbox{with }supp(\mathcal{E})\subseteq B_{\Lambda_{0}R}.
    \end{equation*}
    Near the origin, we have
    \begin{equation}\label{eq. error of s-harmonic near the origin}
        \mathcal{E}(x)=O(|x|^{m+4})\mbox{ with }\|D^{m+4}\mathcal{E}\|_{L^{\infty}(B_{R})}\leq\frac{C\|Q\|_{L^{\infty}(B_{R})}}{R^{m+4}}.
    \end{equation}

    By observing that
    \begin{equation*}
        \widetilde{Q}(x)=\left\{\begin{aligned}
            &Q,&\mbox{if }&x\in B_{R},\\
            &v-Q_{2},&\mbox{if }&x\in\Omega\setminus B_{R},\\
            &0,&\mbox{if }&x\in\Omega^{c},
        \end{aligned}\right.
    \end{equation*}
    we make the following decomposition for $x\in B_{R/2}$:
    \begin{align*}
        (-\Delta)^{s}\widetilde{Q}(x)=&C_{n,s}\int_{\Omega\setminus B_{R}}\frac{Q(x)-(v-Q_{2})(y)}{|x-y|^{n+2s}}dy+(-\Delta)^{s}(Q\cdot\chi_{B_{R}})(x)\\
        =&C_{n,s}\int_{\Omega\setminus B_{R}}\frac{Q(x)}{|x-y|^{n+2s}}dy-C_{n,s}\int_{\Omega\setminus B_{R}}\frac{(v-Q_{2})(y)}{|x-y|^{n+2s}}dy+(-\Delta)^{s}\mathcal{E}(x)\\
        =&C_{n,s}\int_{\Omega\setminus B_{R}}\frac{Q(x)}{|x-y|^{n+2s}}dy-C_{n,s}\int_{\Omega\setminus B_{R}}\frac{(v-Q_{2})(y)}{|x-y|^{n+2s}}dy\\
        &+C_{n,s}\int_{B_{\Lambda_{0}R}\setminus B_{R}}\frac{\mathcal{E}(x)}{|x-y|^{n+2s}}dy-C_{n,s}\int_{B_{\Lambda_{0}R}\setminus B_{R}}\frac{\mathcal{E}(y)}{|x-y|^{n+2s}}dy\\
        &+C_{n,s}P.V.\int_{B_{R}}\frac{\mathcal{E}(x)-\mathcal{E}(y)}{|x-y|^{n+2s}}dy\\
        =:&J_{1}(x)-J_{2}(x)+J_{3}(x)-J_{4}(x)+J_{5}(x).
    \end{align*}
    
    \textbf{Regularity of $J_{1}(x)$.} Notice that when $(x,y)\in B_{R/2}\times (B_{R}^c)$, the kernel $\frac{1}{|x-y|^{n+2s}}$ is analytic both in $x$ and in $y$. Precisely, we have
    \begin{equation*}
        \Big|D_{x}^{l}\frac{1}{|x-y|^{n+2s}}\Big|\leq\frac{C(n,s,m)}{|x-y|^{n+2s+l}},\quad\mbox{for all }0\leq l\leq m+2.
    \end{equation*}
    Using the inequality
    \begin{equation*}
        R^{l}\|D^{l}Q\|_{L^{\infty}(B_{R})}\leq C(n,l)\|Q\|_{L^{\infty}(B_{R})},
    \end{equation*}
    we then infer from the chain rule that
    \begin{equation*}
        \Big|D_{x}^{l}\frac{Q(x)}{|x-y|^{n+2s}}\Big|\leq C(n,s,m)\cdot\sum_{i=0}^{l}\frac{R^{i-l}\|Q\|_{L^{\infty}(B_{R})}}{|x-y|^{n+2s+i}},\quad\mbox{for all }0\leq l\leq m+2.
    \end{equation*}
    Therefore, by integration (which commutes with differentiation), we have
    \begin{equation*}
        |D^{l}J_{1}(x)|\leq C\cdot\sum_{i=0}^{l}\int_{\Omega\setminus B_{R}}\frac{R^{i-l}\|Q\|_{L^{\infty}(B_{R})}}{|x-y|^{n+2s+i}}dy\leq\frac{C(n,s,m)}{R^{l+2s}}\cdot\|Q\|_{L^{\infty}(B_{R})},\quad\mbox{for all }0\leq l\leq m+2.
    \end{equation*}

    \textbf{Regularity of $J_{2}(x)$.} Similarly, by observing that
    \begin{equation*}
        \Big|D_{x}^{l}\frac{(v-Q_{2})(y)}{|x-y|^{n+2s}}\Big|\leq\frac{C(n,s,m)\cdot\|v-Q_{2}\|_{L^{\infty}(\Omega)}}{|x-y|^{n+2s+l}},\quad\mbox{for all }0\leq l\leq m+2,
    \end{equation*}
    we obtain from integration that
    \begin{equation*}
        |D^{l}J_{2}(x)|\leq C\cdot\int_{\Omega\setminus B_{R}}\frac{\|v-Q_{2}\|_{L^{\infty}(\Omega)}}{|x-y|^{n+2s+l}}dy\leq\frac{C(n,s,m)}{R^{l+2s}}\cdot\|v-Q_{2}\|_{L^{\infty}(\Omega)},\quad\mbox{for all }0\leq l\leq m+2.
    \end{equation*}

    \textbf{Regularity of $J_{3}(x)$.} Similar to the analysis of $J_{1}(x)$, we have
    \begin{equation*}
        |D^{l}J_{3}(x)|\leq C\cdot\sum_{i=0}^{l}\int_{B_{\Lambda_{0}R}\setminus B_{R}}\frac{R^{i-l}\|\mathcal{E}\|_{L^{\infty}(B_{R})}}{|x-y|^{n+2s+i}}dy\leq\frac{C(n,s,m)}{R^{l+2s}}\cdot\|Q\|_{L^{\infty}(B_{R})},\quad\mbox{for all }0\leq l\leq m+2.
    \end{equation*}

    \textbf{Regularity of $J_{4}(x)$.} Similar to the analysis of $J_{2}(x)$, we have
    \begin{equation*}
        |D^{l}J_{4}(x)|\leq C\cdot\int_{B_{\Lambda_{0}R}\setminus B_{R}}\frac{\|\mathcal{E}\|_{L^{\infty}(B_{R})}}{|x-y|^{n+2s+l}}dy\leq\frac{C(n,s,m)}{R^{l+2s}}\cdot\|Q\|_{L^{\infty}(B_{R})},\quad\mbox{for all }0\leq l\leq m+2.
    \end{equation*}

    \textbf{Further decompose $J_{5}(x)$.} Let $x\in B_{\delta}$ for $\delta\leq\frac{R}{4}$. We further decompose $J_{5}(x)$ into:
    \begin{align*}
        J_{5}(x)=&C_{n,s}\int_{B_{R}\setminus B_{2\delta}}\frac{\mathcal{E}(x)}{|x-y|^{n+2s}}dy-C_{n,s}\int_{B_{R}\setminus B_{2\delta}}\frac{\mathcal{E}(y)}{|x-y|^{n+2s}}dy+C_{n,s}P.V.\int_{B_{2\delta}}\frac{\mathcal{E}(x)-\mathcal{E}(y)}{|x-y|^{n+2s}}dy\\
        =:&J_{51}(x)-J_{52}(x)+J_{53}(x)
    \end{align*}

    \textbf{Regularity of $J_{51}(x)$.} By \eqref{eq. error of s-harmonic near the origin}, we see that for $x\in B_{\delta}$,
    \begin{equation*}
        |D^{l}\mathcal{E}(x)|\leq C(n,s,m)\frac{\delta^{m+4-l}}{R^{m+4}}\cdot\|Q\|_{L^{\infty}(B_{R})},\quad\mbox{for all }0\leq l\leq m+4.
    \end{equation*}
    Then, similar to the analysis of $J_{1}(x)$, we have
    \begin{align*}
        |D_{x}^{l}J_{51}(x)|\leq&C\frac{\|Q\|_{L^{\infty}(B_{R})}}{R^{m+4}}\cdot\sum_{i=0}^{l}\int_{B_{R}\setminus B_{2\delta}}\frac{\delta^{m+4+i-l}}{|x-y|^{n+2s+i}}dy\leq C\frac{\delta^{m+4-l-2s}}{R^{m+4}}\|Q\|_{L^{\infty}(B_{R})}\\
        \leq&\frac{C(n,s,m)}{R^{l+2s}}\|Q\|_{L^{\infty}(B_{R})},\quad\mbox{for all }0\leq l\leq m+2.
    \end{align*}

    \textbf{Regularity of $J_{52}(x)$.} By \eqref{eq. error of s-harmonic near the origin}, we see that for   $x\in B_\delta$ and $y\in B_{R}\backslash B_{2\delta}$ (with $\delta \leq \frac{R}{4}$):
    \begin{equation*}
        |\mathcal{E}(y)|\leq C(n,s,m)\frac{|y|^{m+4}}{R^{m+4}}\cdot\|Q\|_{L^{\infty}(B_{R})}\leq C(n,s,m)\frac{|x-y|^{m+4}}{R^{m+4}}\cdot\|Q\|_{L^{\infty}(B_{R})}.
    \end{equation*}
    Then, similar to the analysis of $J_{2}(x)$, we have (by noticing $m+4-l\geq2$) that:
    \begin{align*}
        |D_{x}^{l}J_{52}(x)|\leq&C\frac{\|Q\|_{L^{\infty}(B_{R})}}{R^{m+4}}\cdot\int_{B_{R}\setminus B_{2\delta}}\frac{|x-y|^{m+4}}{|x-y|^{n+2s+l}}dy\\
        \leq&\frac{C(n,s,m)}{R^{l+2s}}\|Q\|_{L^{\infty}(B_{R})},\quad\mbox{for all }0\leq l\leq m+2.
    \end{align*}

    \textbf{$L^{\infty}$ estimate of $J_{53}(x)$.} As $B_{\delta}(x)\subseteq B_{2\delta}$, we have 
    \begin{align*}
        |J_{53}(x)|\leq&C_{n,s}\Big|P.V.\int_{B_{\delta}(x)}\frac{\mathcal{E}(x)-\mathcal{E}(y)}{|x-y|^{n+2s}}dy\Big|+C_{n,s}\int_{B_{2\delta}\setminus B_{\delta}(x)}\frac{|\mathcal{E}(x)-\mathcal{E}(y)|}{|x-y|^{n+2s}}dy\\
        \leq&C\int_{B_{\delta}(x)}\frac{|D^{2}\mathcal{E}|_{L^{\infty}(B_{2\delta})}|x-y|^{2}}{|x-y|^{n+2s}}dy+C\int_{B_{2\delta}\setminus B_{\delta}(x)}\frac{|\nabla\mathcal{E}|_{L^{\infty}(B_{2\delta})}|x-y|}{|x-y|^{n+2s}}dy\\
        \leq&C\frac{\delta^{m+2}\|Q\|_{L^{\infty}(B_{R})}}{R^{m+4}}\int_{B_{\delta}(x)}\frac{|x-y|^{2}dy}{|x-y|^{n+2s}}+C\frac{\delta^{m+3}\|Q\|_{L^{\infty}(B_{R})}}{R^{m+4}}\int_{B_{2\delta}\setminus B_{\delta}(x)}\frac{|x-y|dy}{|x-y|^{n+2s}}\\
        \leq&C\frac{\delta^{m+4-2s}\|Q\|_{L^{\infty}(B_{R})}}{R^{m+4}}\leq C\|Q\|_{L^{\infty}(B_{R})}\frac{\delta^{m+2}}{R^{m+2+2s}}.
    \end{align*}

\textbf{Conclusion: construction of the polynomial $P_\delta$.}
From the estimates above, we know that for all $0\leq l\leq m+2$ and $x\in B_{\delta}$,
\begin{equation*}
    \Big|D^{l}\Big((-\Delta)^{s}\widetilde{Q}(x)-J_{53}(x)\Big)\Big|\leq\frac{C(n,s,m)}{R^{l+2s}}\Big(\|Q\|_{L^{\infty}(B_{R})}+\|v-Q_{2}\|_{L^{\infty}(\Omega)}\Big).
\end{equation*}

We now define $P_{\delta}$ as the Taylor polynomial of order $m+1$ of $(-\Delta)^{s}\widetilde{Q}-J_{53}$ at the origin. By Taylor's expansion, for all $x\in B_{\delta}$,
\begin{equation*}
    |(-\Delta)^{s}\widetilde{Q}(x)-P_{\delta}(x)|\leq C\sup_{|\alpha|=m+2}\Big\|D^{\alpha}\Big((-\Delta)^{s}\widetilde{Q}-J_{53}\Big)\Big\|_{L^{\infty}(B_{\delta})}|x|^{m+2}+|J_{53}(x)|.
\end{equation*}
Using the estimates above, we obtain
\begin{equation*}
    \|(-\Delta)^{s}\widetilde{Q}-P_{\delta}\|_{L^{\infty}(B_{\delta})}\leq C(n,s,m)\frac{\delta^{m+2}}{R^{m+2+2s}}\Big(\|Q\|_{L^{\infty}(B_{R})}+\|v-Q_{2}\|_{L^{\infty}(\Omega)}\Big).
\end{equation*}
On the other hand, from the bounds on the derivatives at the origin, we also have
\begin{equation*}
    \|P_{\delta}\|_{L^{\infty}(B_{R})}\leq\frac{C(n,s,m)}{R^{2s}}\Big(\|Q\|_{L^{\infty}(B_{R})}+\|v-Q_{2}\|_{L^{\infty}(\Omega)}\Big).
\end{equation*}
Combining the above estimates yields
\begin{equation*}
    \|P_{\delta}\|_{L^{\infty}(B_{R})}+\left(\frac{R}{\delta}\right)^{m+2}\|(-\Delta)^{s}\widetilde{Q}-P_{\delta}\|_{L^{\infty}(B_{\delta})}\leq\frac{C(n,s,m)}{R^{2s}}\Big(\|Q\|_{L^{\infty}(B_{R})}+\|v-Q_{2}\|_{L^{\infty}(\Omega)}\Big).
\end{equation*}
This completes the proof.
\end{proof}

\subsection{A barrier function}
Next, we  construct a useful barrier function.
\begin{lemma}\label{lem. barrier}
Let $s\in(0,1)$, $\sigma=\min\{s,1-s\}$, and let $R\leq1$. Then, there exist two constants $A,B>0$ depending only on $(n,s)$, such that the following function
\begin{equation*}
w(x) =R^{2}\cdot\Big\{A\cdot(1-\frac{|x|^{2}}{R^{2}})_+^\sigma+B\cdot(1-\frac{|x|^{2}}{R^{2}})_+^s\Big\}
\end{equation*}
satisfies
\begin{equation*}
    (-\Delta)^s w(x)-\Delta w(x)\geq R^{2s}\cdot\delta(x)^{-2s}\mbox{ in }B_{R},\quad\mbox{where }\delta=\delta(x):= R - |x|.
\end{equation*}
\end{lemma}

\begin{proof}
For simplicity, we write
\begin{equation*}
    w_1(x)=(1-\frac{|x|^{2}}{R^{2}})_+^\sigma,\quad\mbox{and }w_2(x)=(1-\frac{|x|^{2}}{R^{2}})_+^s.
\end{equation*}
By direct computation, one can verify that
\begin{equation*}
    [w_{1}]_{C^{\sigma}(\mathbb{R}^{n})}\leq CR^{-\sigma},\quad\mbox{and }|D^{2}w_{1}(x)|\leq C R^{-\sigma}\delta(x)^{\sigma-2}.
\end{equation*}
Besides, the following estimates for all $x\in B_{R}$ are easy to verify:
\begin{equation*}
    -\Delta w_1(x)\geq C^{-1}\cdot R^{-\sigma}\cdot\delta(x)^{\sigma-2},\quad-\Delta w_{2}(x)\geq0,\quad\mbox{and }(-\Delta)^{s}w_{2}(x)\geq C^{-1}R^{-2s}.
\end{equation*}

The most difficult part is to estimate $(-\Delta)^{s}w_{1}$. Let us make the following decomposition:
\begin{align*}
(-\Delta)^{s}w_{1}(x)=&C_{n,s}\,\mathrm{P.V.}\int_{B_{\delta/2}(x)} \frac{w_1(x)-w_1(y)}{|x-y|^{n+2s}}\,dy+C_{n,s}\int_{B_{2R}(x)\setminus B_{\delta/2}(x)}\frac{w_1(x)-w_1(y)}{|x-y|^{n+2s}}\,dy\\
&+C_{n,s}\int_{\mathbb{R}^n \backslash B_{2R}(x)}\frac{w_1(x)-w_{1}(y)}{|x-y|^{n+2s}}\,dy=:\tilde I_1+\tilde I_2+\tilde I_3.
\end{align*}
For $\tilde I_1$, we use the second derivative estimate of $w_{1}$ and get that:
\begin{equation*}
    |\tilde I_1|\leq C\int_{B_{\delta/2}(x)} \frac{\|D^2w_1\|_{L^{\infty}(B_{\delta/2}(x))}\cdot|x-y|^2}{|x-y|^{n+2s}}dy\leq C R^{-\sigma}\delta^{\sigma-2}\cdot\int_{B_{\delta/2}(x)}\frac{|x-y|^{2}dy}{|x-y|^{n+2s}}\leq C R^{-\sigma}\delta^{\sigma-2s}.
\end{equation*}
Applying H\"{o}lder continuity of $w_{1}$, we have (recall that $\sigma<2s$):
\begin{equation*}
    |\tilde I_2|\leq C\int_{B_{2R}(x)\setminus B_{\delta/2}(x)}\frac{[w_{1}]_{C^{\sigma}(\mathbb{R}^{n})}|x-y|^{\sigma}}{|x-y|^{n+2s}}\,dy\leq CR^{-\sigma}\cdot\int_{B_{2R}(x)\setminus B_{\delta/2}(x)}\frac{|x-y|^{\sigma}dy}{|x-y|^{n+2s}}\leq CR^{\sigma}\delta^{-\sigma-2s}.
\end{equation*}
Since $w_{1}\equiv0$ in $\mathbb{R}^n \backslash B_{2R}(x)$, we have:
\begin{equation*}
    |\tilde I_3|=C_{n,s}\int_{\mathbb{R}^n \backslash B_{2R}(x)}\frac{|w_1(x)|}{|x-y|^{n+2s}}\,dy\leq CR^{-\sigma}\delta^{\sigma}\cdot\int_{\mathbb{R}^n \backslash B_{2R}(x)}\frac{dy}{|x-y|^{n+2s}}\leq CR^{-\sigma-2s}\delta^{\sigma}.
\end{equation*}
In total, we conclude that
\begin{equation*}
    |(-\Delta)^{s}w_{1}(x)|\leq CR^{-\sigma}\delta(x)^{\sigma-2s},\quad\mbox{for all }x\in B_{R}.
\end{equation*}
By the Young's inequality, this implies the following estimate for sufficiently small $\epsilon$:
\begin{equation*}
    (-\Delta)^{s}w_{1}(x)\geq-\epsilon R^{2-2s-\sigma}\delta(x)^{\sigma-2}-C(\epsilon,s,\sigma)R^{-2s},
\end{equation*}
where $\displaystyle\lim_{\epsilon\to0}C(\epsilon,s,\sigma)=\infty$.

For $w(x)$, we have
\begin{equation*}
    (-\Delta)^{s}w(x)-\Delta w(x)\geq A\cdot\Big\{\frac{R^{2-\sigma}\cdot\delta(x)^{\sigma-2}}{C}-\epsilon R^{4-2s-\sigma}\delta(x)^{\sigma-2}-C(\epsilon,s,\sigma)R^{2-2s}\Big\}+\frac{B\cdot R^{2-2s}}{C}.
\end{equation*}
By choosing $\epsilon$ small, we have
\begin{equation*}
    \frac{R^{2-\sigma}\cdot\delta(x)^{\sigma-2}}{C}-\epsilon R^{4-2s-\sigma}\delta(x)^{\sigma-2}=(\frac{1}{C}-\epsilon R^{2-2s})\cdot R^{2-\sigma}\cdot\delta(x)^{\sigma-2}\geq\frac{1}{2C}\cdot R^{2-\sigma}\cdot\delta(x)^{\sigma-2}.
\end{equation*}
By letting the ratio $\frac{B}{A}$ be sufficiently large, we have
\begin{equation*}
    A\cdot C(\epsilon)R^{2-2s}\leq\frac{B\cdot R^{2-2s}}{C}.
\end{equation*}
Finally, we let $A$ be sufficiently large. Using the fact $\sigma\leq2-2s$, we have
\begin{equation*}
    (-\Delta)^{s}w(x)-\Delta w(x)\geq\frac{A}{2C}\cdot R^{2-\sigma}\cdot\delta(x)^{\sigma-2}\geq R^{2s}\cdot\delta(x)^{-2s}.
\end{equation*}
This proves that $w(x)$ is the desired upper solution.
\end{proof}

\section{Campanato iteration}
In this section, we use the Campanato iteration to obtain the pointwise Schauder estimate.
\subsection{Outline} We first make the following settings and notations:
    \begin{itemize}
        \item Let $\lambda\in(0,1)$, whose value will be decided later. Let $r_{k}:=\lambda^{k}r$. Denote $B_\ast:=B_\ast(0),~\ast=r,r_k$ or $ {r_k}/{2}$.
        
        \item We make the following abbreviation for the right-hand side information:
        \begin{equation*}
            \mathcal{A}:=\|u\|_{L^{\infty}(B_{r})}+r^{2}\cdot\int_{B_{r}^{c}}\frac{|u(x)|}{|x|^{n+2s}}dx+r^{2}\cdot\|F\|_{L^{\infty}(B_{r})}.
        \end{equation*}
        \item Let $\{P_{k}(x)\}_{k\geq0}$ be a sequence of $(m+2)$-order polynomials. As a start, we let $P_{0}(x)\equiv0$. Moreover, we let $\widetilde{P}_{k}(x):=P_{k}(x)\cdot\chi_{B_{r_{k}}}+u(x)\cdot\chi_{B_{r_{k}}^{c}}$.
        \item Let $G_{k}(x)$ be the $m$-th order Taylor polynomial of $(-\Delta)^{s}\widetilde{P}_{k}$ centered at the origin.
        \item For each $k\geq0$, let $h_{k}(x)$ be the solution of the following problem:
        \begin{equation}\label{eq. Dirichlet problem}
            \left\{\begin{aligned}
                &\Delta h_{k}(x)=G_{k}(x)-F(x)&\mbox{in }&B_{r_{k}/2},\\
                &h_{k}(x)=u(x)&\mbox{on }&\partial B_{r_{k}/2}.
            \end{aligned}\right.
        \end{equation}
        Let $\widetilde{h}_{k}(x):=h_{k}(x)\cdot\chi_{B_{r_{k}/2}}+u(x)\cdot\chi_{B_{r_{k}/2}^{c}}$.
        \item For each $k\geq1$, we define $P_{k}(x)$ as the $(m+2)$-order Taylor polynomial of $h_{k-1}(x)$ centered at the origin.
        \item We define $\epsilon_{k}$ for each $k\geq0$ as follows:
        \begin{equation*}
            \epsilon_{k}:=\|u(x)-P_{k}(x)\|_{L^{\infty}(B_{r_{k}})}.
        \end{equation*}
        Clearly, we have $\epsilon_{0}\leq\mathcal{A}$.
    \end{itemize}
    A constant $C$ is called uniform if it depends only on $(n,s,m,\alpha)$, but is independent of $(\lambda,k)$ and $u(x)$. During the proof, the uniform constant $C$ is allowed to change from line to line.

\subsection{The Dirichlet problem}
To understand the solution $h_{k}$ to the Dirichlet problem \eqref{eq. Dirichlet problem}, we must first estimate the effect of $(-\Delta)^{s}$ acting on polynomials.
\begin{lemma}\label{lem. (Delta)^s for tilde P_0}
    We have the following estimate of $(-\Delta)^{s}\widetilde{P}_{0}$ in $B_{r/2}$:
    \begin{equation*}
            r^{l}\Big\|D^{l}(-\Delta)^{s}\widetilde{P}_{0}\Big\|_{L^{\infty}(B_{r/2})}\leq\frac{C\cdot\mathcal{A}}{r^{2}},\quad\mbox{for all }0\leq l\leq m+1.
    \end{equation*}
\end{lemma}
\begin{proof}
    By noticing $\widetilde{P}_{0}=u(x)\cdot\chi_{B_{r}^{c}}$, the result is just a matter of simple integration.
\end{proof}

\begin{lemma}\label{lem. (Delta)^s difference in k}
    For each $k\geq1$, we have
    \begin{equation*}
        r_{k}^{l}\Big\|D^{l}(-\Delta)^{s}(\widetilde{P}_{k}-\widetilde{P}_{k-1})\Big\|_{L^{\infty}(B_{r_{k}/2})}\leq\frac{C}{r_{k}^{2s}}\cdot(\epsilon_{k}+\epsilon_{k-1}),\quad\mbox{for all }0\leq l\leq m+1.
    \end{equation*}
\end{lemma}
\begin{proof}
    Let us apply Lemma~\ref{lem. very general (-Delta)^s osc estimate} with
    \begin{equation*}
        Q_{1}=P_{k},\quad Q_{2}=P_{k-1},\quad v=u,\quad R=r_{k},\quad\Omega=B_{r_{k-1}}.
    \end{equation*}
    Then, we have
    \begin{equation*}
        r_{k}^{l}\Big\|D^{l}(-\Delta)^{s}(\widetilde{P}_{k}-\widetilde{P}_{k-1})\Big\|_{L^{\infty}(B_{r_{k}/2})}\leq\frac{C}{r_{k}^{2s}}\cdot\Big(\|P_{k}-P_{k-1}\|_{L^{\infty}(B_{r_{k}})}+\|P_{k-1}-u\|_{L^{\infty}(B_{r_{k-1}})}\Big)
    \end{equation*}
    for all $0\leq l\leq m+1$. Using $|P_{k}-P_{k-1}|\leq|P_{k}-u|+|P_{k-1}-u|$ and recalling the definition of $\epsilon_{k}$, we get the desired estimate.
\end{proof}

\begin{corollary}\label{cor. Lipschitz norm of (-Delta)^s after summation}
    For each $k\geq0$, we have
    \begin{equation*}
        \Big\|D^{m+1}(-\Delta)^{s}\widetilde{P}_{k}\Big\|_{L^{\infty}(B_{r_{k}/2})}\leq C\cdot\mathcal{O}_{k},\quad\mbox{where }\mathcal{O}_{k}:=\frac{\mathcal{A}}{r^{m+3}}+\sum_{j=0}^{k}\frac{\epsilon_{j}}{r_{j+1}^{m+1+2s}}.
    \end{equation*}
\end{corollary}
\begin{proof}
    By combining Lemma~\ref{lem. (Delta)^s for tilde P_0} with Lemma~\ref{lem. (Delta)^s difference in k}, we have for each $k\geq0$:
    \begin{equation*}
        \Big\|D^{m+1}(-\Delta)^{s}\widetilde{P}_{k}\Big\|_{L^{\infty}(B_{r_{k}/2})}\leq\frac{C\cdot\mathcal{A}}{r^{m+3}}+\sum_{j=1}^{k}\frac{C}{r_{j}^{m+1+2s}}\cdot(\epsilon_{j}+\epsilon_{j-1}).
    \end{equation*}
    Then, by using the facts $\epsilon_{0}\leq\mathcal{A}$, $r_{j}\geq r_{j+1}$, and $r\leq1$, we obtain the desired estimate.
\end{proof}

Next, we obtain the following $C^{m+3}$ estimate of $h_{k}$.
\begin{lemma}\label{lem. C3 estimate for the local Dirichlet problem}
    For all $k\geq0$, and every $x\in B_{r_{k}/2}$, we write $\delta=dist(x,\partial B_{r_{k}/2})$. Then, for every $0\leq l\leq m+3$, we have the following estimate for $h_{k}$, which is the solution to \eqref{eq. Dirichlet problem}:
    \begin{equation*}
        \delta^{l}|D^{l}(h_{k}-P_{k})(x)|\leq\left\{\begin{aligned}
            &C\cdot(\epsilon_{k}+r_{k}^{2-2s}\epsilon_{k-1}),&\mbox{when }&k\geq1,\\
            &C\cdot\mathcal{A},&\mbox{when }&k=0.
        \end{aligned}\right.
    \end{equation*}
\end{lemma}
\begin{proof}
    Notice that $\Delta P_{k}(x)$ is equal to the $m$-order Taylor polynomial of $\Delta h_{k-1}(x)$ centered at the origin, we must have $\Delta P_{k}(x)=G_{k-1}(x)-F(x)$ for all $k\geq1$. Then, for all $x\in B_{r_{k}/2}$,
    \begin{equation*}
        \Delta(h_{k}-P_{k})(x)=\left\{\begin{aligned}
            &G_{k}(x)-G_{k-1}(x),&\mbox{when }&k\geq1,\\
            &\Delta h_{0}(x)=G_{0}(x)-F(x),&\mbox{when }&k=0.
        \end{aligned}\right.
    \end{equation*}
    The right-hand side is a $m$-order polynomial, whose estimate was previously given in Lemma~\ref{lem. (Delta)^s for tilde P_0} and Lemma~\ref{lem. (Delta)^s difference in k}. Besides, on $\partial B_{r_{k}/2}$, we have
    \begin{equation*}
        |(h_{k}-P_{k})(x)|=|(u-P_{k})(x)|\leq\epsilon_{k}.
    \end{equation*}
    Then we apply the $C^{m+3}$ estimate for \eqref{eq. Dirichlet problem}, and the desired estimate follows.
\end{proof}

\begin{lemma}\label{lem. rhs for h_k}
    Let $x\in B_{r_{k}/2}$, with $\delta=dist(x,\partial B_{r_{k}/2})$. Then,
    \begin{equation*}
        \Big|(-\Delta)^{s}\widetilde{h}_{k}-\Delta\widetilde{h}_{k}-F(x)\Big|\leq\left\{\begin{aligned}
            &\frac{C}{\delta^{2s}}\cdot(\epsilon_{k}+r_{k}^{2-2s}\epsilon_{k-1})+C r_{k}^{m+1}\cdot\mathcal{O}_{k},&\mbox{when }&k\geq1,\\
            &\frac{C}{\delta^{2s}}\cdot\mathcal{A}+C r_{k}^{m+1}\cdot\mathcal{O}_{k},&\mbox{when }&k=0.
        \end{aligned}\right.
    \end{equation*}
\end{lemma}
\begin{proof}
    Let $x\in B_{r_{k}/2}$ be arbitrary, and write $\delta=dist(x,\partial B_{r_{k}/2})$. Let $T(y)$ be the $(m+2)$-order Taylor expansion of $h_{k}(y)$ centered at $x$, and let
    \begin{equation*}
        \widetilde{T}(y)=T(y)\cdot\chi_{B_{\delta/2}(x)}+u(y)\cdot\chi_{B_{\delta/2}^{c}(x)}.
    \end{equation*}
    From the construction of $h_{k}$, we have
    \begin{align*}
        &(-\Delta)^{s}\widetilde{h}_{k}(x)-\Delta\widetilde{h}_{k}(x)-F(x)=(-\Delta)^{s}\widetilde{h}_{k}(x)-G_{k}(x)\\
        =&(-\Delta)^{s}(\widetilde{h}_{k}-\widetilde{T})(x)+(-\Delta)^{s}(\widetilde{T}-\widetilde{P}_{k})(x)+\Big[(-\Delta)^{s}\widetilde{P}_{k}(x)-G_{k}(x)\Big]\\
        =:&I_{1}+I_{2}+I_{3}.
    \end{align*}
    
    \textbf{Estimate of $I_{1}$.} One easily notice that $(\widetilde{h}_{k}-\widetilde{T})$ is supported in $B_{r_{k}/2}$, so
    \begin{align*}
        I_{1}=&\int_{B_{\delta/2}(x)}\frac{(\widetilde{h}_{k}-\widetilde{T})(x)-(\widetilde{h}_{k}-\widetilde{T})(y)}{|x-y|^{n+2s}}dy+\int_{B_{r_{k}/2}\setminus B_{\delta/2}(x)}\frac{(\widetilde{h}_{k}-\widetilde{P}_{k})(x)-(\widetilde{h}_{k}-\widetilde{P}_{k})(y)}{|x-y|^{n+2s}}dy\\
        &+\int_{B_{r_{k}/2}\setminus B_{\delta/2}(x)}\frac{(\widetilde{P}_{k}-\widetilde{T})(x)-(\widetilde{P}_{k}-\widetilde{T})(y)}{|x-y|^{n+2s}}dy=: I_{11}+I_{12}+I_{13}.
    \end{align*}
    As $D^{m+3}P_{k}\equiv D^{m+3}T\equiv0$, we can estimate $I_{11}$ by
    \begin{equation*}
        |I_{11}|\leq C\cdot\delta^{m+3-2s}\cdot\|D^{m+3}(h_{k}-P_{k})\|_{L^{\infty}(B_{\delta/2}(x))}
    \end{equation*}
    As for $I_{12}$, we have
    \begin{equation*}
        |I_{12}|\leq C\cdot\delta^{1-2s}\cdot\|\nabla(h_{k}-P_{k})\|_{L^{\infty}(B_{\delta/2}(x))}
    \end{equation*}
    As for $I_{3}$, by noticing that $\widetilde{P}_{k}-\widetilde{T}=P_{k}-u$ in $B_{r_{k}/2}\setminus B_{\delta/2}(x)$, we have
    \begin{equation*}
        |I_{13}|\leq C\cdot\delta^{-2s}\cdot\epsilon_{k}.
    \end{equation*}
    Using Lemma~\ref{lem. C3 estimate for the local Dirichlet problem}, we then have
    \begin{equation*}
        |I_{1}|\leq\left\{\begin{aligned}
            &\frac{C}{\delta^{2s}}\cdot(\epsilon_{k}+r_{k}^{2-2s}\epsilon_{k-1}),&\mbox{when }&k\geq1,\\
            &\frac{C}{\delta^{2s}}\cdot\mathcal{A},&\mbox{when }&k=0.
        \end{aligned}\right.
    \end{equation*}

    \textbf{Estimate of $I_{2}$.} In Lemma~\ref{lem. very general (-Delta)^s osc estimate}, we set
    \begin{equation*}
        Q_{1}(y)=T(y-x),\quad Q_{2}(y)=P_{k}(y-x),\quad v(y)=u(y-x),\quad R=\delta/2,\quad\Omega=B_{r_{k}}(-x).
    \end{equation*}
    Using the fact that $\|Q_{2}-v\|_{L^{\infty}(\Omega)}=\|P_{k}-u\|_{L^{\infty}(B_{r_{k}})}=\epsilon_{k}$, we have
    \begin{equation*}
        |I_{2}|\leq\frac{C}{\delta^{2s}}\cdot\Big(\|Q\|_{L^{\infty}(B_{R})}+\|v-Q_{2}\|_{L^{\infty}(\Omega)}\Big)=\frac{C}{\delta^{2s}}\cdot\Big(\|T-P_{k}\|_{L^{\infty}(B_{\delta/2}(x))}+\epsilon_{k}\Big).
    \end{equation*}
    The term $\|T-P_{k}\|_{L^{\infty}(B_{\delta/2}(x))}$ can be estimated by Lemma~\ref{lem. C3 estimate for the local Dirichlet problem}, so
    \begin{equation*}
        |I_{2}|\leq\left\{\begin{aligned}
            &\frac{C}{\delta^{2s}}\cdot(\epsilon_{k}+r_{k}^{2-2s}\epsilon_{k-1}),&\mbox{when }&k\geq1,\\
            &\frac{C}{\delta^{2s}}\cdot\mathcal{A},&\mbox{when }&k=0.
        \end{aligned}\right.
    \end{equation*}

    \textbf{Estimate of $I_{3}$.} Since $x\in B_{r_{k}/2}$, we have $|I_{3}|\leq\|D^{m+1}(-\Delta)^{s}\widetilde{P}_{k}\|_{L^{\infty}(B_{r_{k}/2})}\cdot r_{k}^{m+1}$. Then, by Corollary~\ref{cor. Lipschitz norm of (-Delta)^s after summation},
    \begin{equation*}
        |I_{3}|\leq C\cdot r_{k}^{m+1}\cdot\mathcal{O}_{k}.
    \end{equation*}

    Combining the estimates for $I_{1}$-$I_{3}$ gives the desired estimate.
\end{proof}

\begin{lemma}\label{L6}
    We have the following estimate for $\epsilon_{k+1}=\|u-P_{k+1}\|_{L^{\infty}(B_{r_{k+1}})}:$
    \begin{equation*}
        \epsilon_{k+1}\leq\left\{\begin{aligned}
            &C(r_{k}^{2-2s}+\lambda^{m+3})\cdot(\epsilon_{k}+r_{k}^{2-2s}\epsilon_{k-1})+C\cdot r_{k}^{m+3}\cdot\mathcal{O}_{k}+Cr_{k}^{m+2}\omega(r_{k}),&\mbox{when }&k\geq1,\\
            &C(r_{k}^{2-2s}+\lambda^{m+3})\cdot\mathcal{A}+C\cdot r_{k}^{m+3}\cdot\mathcal{O}_{k}+Cr_{k}^{m+2}\omega(r_{k}),&\mbox{when }&k=0.
        \end{aligned}\right.
    \end{equation*}
\end{lemma}
\begin{proof}
    We decompose $\|u-P_{k+1}\|_{L^{\infty}(B_{r_{k+1}})}$ into the following two parts:
    \begin{equation*}
        \|u-P_{k+1}\|_{L^{\infty}(B_{r_{k+1}})}\leq\|u-h_{k}\|_{L^{\infty}(B_{r_{k}/2})}+\|h_{k}-P_{k+1}\|_{L^{\infty}(B_{r_{k+1}})}.
    \end{equation*}
    By Lemma~\ref{lem. rhs for h_k}, we have the following estimate for $\{(-\Delta)^{s}-\Delta\}(u-h_{k})$: Let $x\in B_{r_{k}/2}$, with $\delta=dist(x,\partial B_{r_{k}/2})$, then
    \begin{equation*}
        \Big|\{(-\Delta)^{s}-\Delta\}(u-h_{k})\Big|(x)\leq\left\{\begin{aligned}
            &\frac{C}{\delta^{2s}}\cdot(\epsilon_{k}+r_{k}^{2-2s}\epsilon_{k-1})+C r_{k}^{m+1}\cdot\mathcal{O}_{k}+\omega(r_{k})r_{k}^{m},&\mbox{when }&k\geq1,\\
            &\frac{C}{\delta^{2s}}\cdot\mathcal{A}+C r_{k}^{m+1}\cdot\mathcal{O}_{k}+\omega(r_{k})r_{k}^{m},&\mbox{when }&k=0.
        \end{aligned}\right.
    \end{equation*}
    As $u-h_{k}$ is supported in $B_{r_{k}/2}$ with $r_{k}/2\leq1$, we then apply Lemma~\ref{lem. barrier}, and get that
    \begin{equation*}
        \|u-h_{k}\|_{L^{\infty}(B_{r_{k}/2})}\leq\left\{\begin{aligned}
            &Cr_{k}^{2-2s}\cdot(\epsilon_{k}+r_{k}^{2-2s}\epsilon_{k-1})+C\cdot r_{k}^{m+3}\cdot\mathcal{O}_{k}+Cr_{k}^{m+2}\omega(r_{k}),&\mbox{when }&k\geq1,\\
            &Cr_{k}^{2-2s}\mathcal{A}+C\cdot r_{k}^{m+3}\cdot\mathcal{O}_{k}+Cr_{k}^{m+2}\omega(r_{k}),&\mbox{when }&k=0.
        \end{aligned}\right.
    \end{equation*}
    As $P_{k}$'s are $(m+2)$-order polynomials, we must have $D^{m+3}(h_{k}-P_{k+1})=D^{m+3}(h_{k}-P_{k})$ in $B_{r_{k}/2}$. Then, as $P_{k+1}$ is the $(m+2)$-order Taylor polynomial of $h_{k}$, we must have:
    \begin{equation*}
        \|h_{k}-P_{k+1}\|_{L^{\infty}(B_{r_{k+1}})}\leq\|D^{m+3}(h_{k}-P_{k+1})\|_{L^{\infty}(B_{r_{k+1}})}\cdot r_{k+1}^{m+3}\leq\|D^{m+3}(h_{k}-P_{k})\|_{L^{\infty}(B_{r_{k}/4})}\cdot r_{k+1}^{m+3}.
    \end{equation*}
    By Lemma~\ref{lem. C3 estimate for the local Dirichlet problem}, we have
    \begin{equation*}
        \|h_{k}-P_{k+1}\|_{L^{\infty}(B_{r_{k+1}})}\leq\left\{\begin{aligned}
            &C\lambda^{m+3}(\epsilon_{k}+r_{k}^{2-2s}\epsilon_{k-1}),&\mbox{when }&k\geq1,\\
            &C\lambda^{m+3}\mathcal{A},&\mbox{when }&k=0.
        \end{aligned}\right.
    \end{equation*}
    By adding up the estimates for $\|u-h_{k}\|_{L^{\infty}(B_{r_{k}/2})}$ and $\|h_{k}-P_{k+1}\|_{L^{\infty}(B_{r_{k+1}})}$, we get the desired estimate.
\end{proof}

\subsection{Proof of the Schauder estimate}
So far, we have used $P_{k}$'s to approximate the solution $u$ in $B_{r_{k}}$, where $P_{k}$ is a $(m+2)$-order polynomial (where $m\geq0$).

Now for each $(m+2)$-order polynomial $P$, we let $\Pi_{l}[P]$ be the $l$-degree term of $P$. This means that $\Pi_{l}[P]$ is a $l$-order homogeneous polynomial and $\displaystyle P=\sum_{l=0}^{m+2}\Pi_{l}[P]$. For such a decomposition, we see that the following inequality holds for all $(m+2)$-order polynomials $P$'s:
\begin{equation}\label{eq. coefficient of polynomial}
    \Big\|\Pi_{l}[P]\Big\|_{L^{\infty}(B_{r})}\leq C(n,m)\lambda^{-kl}\|P\|_{L^{\infty}(B_{r_{k}})},\quad\mbox{for all }k\geq0\mbox{ and }0\leq l\leq m+2.
\end{equation}
Notice that the left-hand side of \eqref{eq. coefficient of polynomial} represents the magnitude of the coefficients of $\Pi_{l}[P]$.

Then, we denote $\Pi_{k,l}:=\Pi_{l}[P_{k}]$. Recall that as we have chosen $P_{0}\equiv0$, we must have $\Pi_{0,l}\equiv0$. We will later apply \eqref{eq. coefficient of polynomial} to obtain the convergence of the approximating polynomials (except possibly for the highest degree term).
\begin{proof}[Proof of Theorem \ref{t1}]
Let us treat $\alpha=0$ in Theorem \ref{t1} (1)(2). Since $s\in(0,1)$, we have $2-2s>0$. We can fix a $\lambda>0$ sufficiently small such that
\begin{equation*}
C\lambda^{2-2s}\leq\frac{1}{8},\quad\mbox{and }C\lambda^{m+3}\leq\frac{1}{8}\lambda^{m+2+\alpha}.
\end{equation*}
For such a fixed uniform $\lambda$, there exists a sufficiently large $K_{0}=K_{0}(n,s,m,\alpha)$, such that
\begin{equation*}
    C r_k^{2-2s}\leq\frac{1}{8}\lambda^{m+2+\alpha}\mbox{ and }\frac{r_{k}^{2-2s}}{\lambda^{m+1+2s}}\leq\frac{1}{8}\lambda^{2m+4+2\alpha},\quad\mbox{for all }k \ge K_0.
\end{equation*}

Let us denote
\begin{equation*}
    \tau_{k}=r_{k}^{m+3}\mathcal{O}_{k},\quad S_k=\tau_k+\lambda^{m+3}\epsilon_{k},\quad\mbox{and }\widetilde{S}_{k}=r_{k}^{-(m+2+\alpha)}S_{k}.
\end{equation*}
We claim that there exists some uniform $q\in(0,1)$, such that
\begin{equation}\label{eq. claim of tilde S_k}
    S_{k}\leq C_{1}\lambda^{k(m+2+\alpha)}\cdot\Big(q^{k}\mathcal{A}+r^{m+2+\alpha}\cdot\sum_{j=0}^{k-1}q^{k-j}\frac{\omega(r_{j})}{r_{j}^{\alpha}}\Big),\quad\mbox{for all }k\geq0.
\end{equation}

To prove \eqref{eq. claim of tilde S_k}, we first notice that $\epsilon_{0}\leq\mathcal{A}$ and $\mathcal{O}_{k}=\frac{\mathcal{A}}{r^{m+3}}+\frac{\epsilon_{0}}{r_{1}^{m+1+2s}}$ imply that $S_{0}\leq C\mathcal{A}$. Next, for $k\leq K_{0}$, we apply Corollary~\ref{cor. Lipschitz norm of (-Delta)^s after summation} and Lemma~\ref{L6} finitely many times, and obtain that \eqref{eq. claim of tilde S_k} holds for $k\leq K_{0}$ by choosing $C_{1}$ sufficiently large.

Now we let $k\geq K_{0}$, by Corollary~\ref{cor. Lipschitz norm of (-Delta)^s after summation} and Lemma~\ref{L6}, we have
\begin{equation*}
    \left\{\begin{aligned}
        &\epsilon_{k+1}\leq\frac{1}{4}\lambda^{m+2+\alpha}\epsilon_{k}+\frac{1}{4}\lambda^{2m+4+2\alpha}\epsilon_{k-1}+C\cdot\tau_{k}+C\cdot r_{k}^{m+2}\omega(r_{k}),\\
        &\tau_{k+1}\leq\frac{1}{8}\lambda^{m+2+\alpha}\tau_{k}+\frac{1}{8}\lambda^{2m+4+2\alpha}\epsilon_{k+1},
    \end{aligned}\right.\quad\mbox{for all }k\geq K_{0}.
\end{equation*}
It follows that
\begin{equation*}
    S_{k+1}\leq\frac{1}{3}\lambda^{m+2+\alpha} S_k+\frac{1}{3}\lambda^{2m+4+2\alpha}S_{k-1}+C\la^{m+3} r_{k}^{m+2}\omega(r_{k}),\quad\mbox{for all }k\geq K_{0}.
\end{equation*}
Dividing both sides by $r_{k+1}^{m+2+\alpha}$ gives $\displaystyle\widetilde{S}_{k+1}\leq\frac{1}{3}\widetilde{S}_{k}+\frac{1}{3}\widetilde{S}_{k-1}+\frac{\omega(r_{k})}{r_{k}^{\alpha}}$. With this, an inductive argument shows that \eqref{eq. claim of tilde S_k} holds for all $k\geq K_{0}$ with $q=0.9$.

From \eqref{eq. claim of tilde S_k}, we see that
\begin{equation}\label{eq. final estimate of epsilon_k}
    \epsilon_{k}\leq C\lambda^{k(m+2+\alpha)}\cdot\Big(q^{k}\mathcal{A}+r^{m+2+\alpha}\cdot\sum_{j=0}^{k-1}q^{k-j}\frac{\omega(r_{j})}{r_{j}^{\alpha}}\Big),\quad\mbox{for all }k\geq0.
\end{equation}
Noticing that $\|P_{k+1}-P_{k}\|_{L^{\infty}(B_{r_{k+1}})}\leq\epsilon_{k+1}+\epsilon_{k}$, it then follows from \eqref{eq. coefficient of polynomial} and \eqref{eq. final estimate of epsilon_k} that
\begin{equation}\label{eq. adjacent coefficient difference}
    \|\Pi_{k+1,l}-\Pi_{k,l}\|_{L^{\infty}(B_{r})}\leq C\lambda^{k(m+2+\alpha-l)}\cdot\Big(q^{k}\mathcal{A}+r^{m+2+\alpha}\cdot\sum_{j=0}^{k-1}q^{k-j}\frac{\omega(r_{j})}{r_{j}^{\alpha}}\Big)
\end{equation}
holds for all $k\geq0$ and $0\leq l\leq m+2$. Now we prove Theorem~\ref{t1} in three cases:

\textbf{Proof of part (1).} Let $\displaystyle\mathcal{S}:=\sup_{\rho\leq r}\omega(\rho)$. Then, it follows from \eqref{eq. final estimate of epsilon_k} and \eqref{eq. adjacent coefficient difference} that
\begin{equation*}
    \epsilon_{k}\leq C\lambda^{k(m+2)}\cdot(q^{k}\mathcal{A}+r^{m+2}\mathcal{S})\mbox{ and }\|\Pi_{k+1,l}-\Pi_{k,l}\|_{L^{\infty}(B_{r})}\leq C\lambda^{k(m+2-l)}\cdot(q^{k}\mathcal{A}+r^{m+2}\mathcal{S}).
\end{equation*}
In particular, by the fact $\Pi_{0,l}\equiv0$, we have
\begin{equation*}
    \|\Pi_{k,l}\|_{L^{\infty}(B_{r})}\leq\left\{\begin{aligned}
        &C\cdot(\mathcal{A}+r^{m+2}\mathcal{S}),&\mbox{if }&0\leq l\leq m+1,\\
        &Ck\cdot(\mathcal{A}+r^{m+2}\mathcal{S}),&\mbox{if }&l=m+2.
    \end{aligned}\right.
\end{equation*}
Moreover, the coefficients of $\Pi_{k,l}$ form a Cauchy sequence (index by $k$) for every $0\leq l\leq m+1$. Then, let $P$ be a $(m+1)$-order polynomial defined as $\displaystyle P=\lim_{k\to\infty}\Big(\sum_{l=0}^{m+1}\Pi_{k,l}\Big)$. It is then obvious that $\|P\|_{L^{\infty}(B_{r})}\leq C\cdot(\mathcal{A}+r^{m+2}\mathcal{S})$. Moreover, for each $k\geq0$, we have
\begin{align*}
    \|u-P\|_{L^{\infty}(B_{r_{k}})}\leq&\epsilon_{k}+\lambda^{k(m+2)}\|\Pi_{k,m+2}\|_{L^{\infty}(B_{r})}+\sum_{j=k}^{\infty}\Big(\sum_{l=0}^{m+1}\lambda^{kl}\|\Pi_{j+1,l}-\Pi_{j,l}\|_{L^{\infty}(B_{r})}\Big)\\
    \leq&C\lambda^{k(m+2)}k\cdot(\mathcal{A}+r^{m+2}\mathcal{S})\leq C\cdot(\mathcal{A}+r^{m+2}\mathcal{S})\cdot(\frac{r_{k}}{r})^{m+2}\ln{(\frac{2r}{r_{k}})}.
\end{align*}

\textbf{Proof of part (2).} Similarly, from \eqref{eq. final estimate of epsilon_k} and \eqref{eq. adjacent coefficient difference}, one has:
\begin{equation*}
     \epsilon_{k}\leq C\lambda^{k(m+2)}\cdot\Big(q^{k}\mathcal{A}+r^{m+2}\cdot\sum_{j=0}^{k-1}q^{k-j}\omega(r_{j})\Big),
\end{equation*}
and
\begin{equation*}
    \|\Pi_{k+1,l}-\Pi_{k,l}\|_{L^{\infty}(B_{r})}\leq C\lambda^{k(m+2-l)}\cdot\Big(q^{k}\mathcal{A}+r^{m+2}\cdot\sum_{j=0}^{k-1}q^{k-j}\omega(r_{j})\Big),\quad\mbox{for all }0\leq l\leq m+2.
\end{equation*}
Then for all $0\leq k_{1}\leq k_{2}$ and all $0\leq l\leq m+2$, it holds that
\begin{equation*}
    \|\Pi_{k_{2},l}-\Pi_{k_{1},l}\|_{L^{\infty}(B_{r})}\leq C\lambda^{k_{1}(m+2-l)}\cdot\Big(q^{k_{1}}\mathcal{A}+r^{m+2}\sum_{j=k_{1}}^{k_{2}-1}\omega(r_{j})+r^{m+2}\cdot\sum_{j=0}^{k_{1}-1}q^{k_{1}-j}\omega(r_{j})\Big).
\end{equation*}
In other words, $\Pi_{k,l}$ forms a Cauchy sequence for all $0\leq l\leq m+2$. Let $\displaystyle P=\lim_{k\to\infty}\Big(\sum_{l=0}^{m+2}\Pi_{k,l}\Big)$ be the limiting $(m+2)$-order polynomial. By letting $k_{1}=0$ and $k_{2}\to\infty$ and by $\Pi_{0,l}\equiv0$, we see
\begin{equation*}
    \|P\|_{L^{\infty}(B_{r})}\leq C\Big(\mathcal{A}+r^{m+2}\sum_{j=0}^{\infty}\omega(r_{j})\Big).
\end{equation*}
Moreover, for each $k\geq0$, a similar argument via the triangle inequality as in part (1) gives
\begin{align*}
    \|u-P\|_{L^{\infty}(B_{r_{k}})}\leq&\epsilon_{k}+\sum_{l=0}^{m+2}\lambda^{kl}\cdot\Big(\lim_{j\to\infty}\|\Pi_{j,l}-\Pi_{k,l}\|_{L^{\infty}(B_{r})}\Big)\\
    \leq&C\lambda^{k(m+2)}\cdot\Big(q^{k}\mathcal{A}+r^{m+2}\sum_{j=k}^{\infty}\omega(r_{j})+r^{m+2}\cdot\sum_{j=0}^{k-1}q^{k-j}\omega(r_{j})\Big)\\
    \leq&C(\frac{r_{k}}{r})^{m+2}\cdot\Big((\frac{r_{k}}{r})^{\mu}\mathcal{A}+r^{m+2}\int_{0}^{r_{k}}\frac{\omega(t)}{t}dt+r^{m+2}r_{k}^{\mu}\int_{r_{k}}^{r}\frac{\omega(t)}{t^{1+\mu}}dt\Big).
\end{align*}
Here, $\mu>0$ is a uniform constant such that $q=\lambda^{\mu}$.

\textbf{Proof of part (3).} The proof is similar to the proof of part (2) and we omit the details.

Then, we have finished the proof of Theorem~\ref{t1}.
\end{proof}

\vspace{2mm}

\noindent \textbf{Acknowledgments.}


P. Wang  was partially supported by National Natural Science Foundation of China (Grant No. 12101530), the Natural Science Foundation of Henan Province (262300421235),  the Sponsored by  Program for Science \& Technology Innovation Talents in Universities of Henan Province (26HASTIT040) and the Nanhu Scholars Program for Young Scholars of XYNU.
 L. Wu was partially supported by National Natural Science Foundation of China (Grant No. 12401133) and the Guangdong Basic and Applied Basic Research Foundation (2025B151502069).

\vspace{2mm}
\noindent \textbf{Conflict of interest.} The authors do not have any possible conflicts of interest.

\vspace{2mm}

\noindent \textbf{Data availability statement.}
Data sharing is not applicable to this article, as no data sets were generated or analyzed during the current study.

\end{document}